\documentclass[11pt, letterpaper]{amsart}
\usepackage[utf8]{inputenc}

\usepackage{amssymb, amsmath, amsthm, wasysym, mathrsfs}
\usepackage{hyperref}
\usepackage{cleveref}
\usepackage[alphabetic,lite]{amsrefs}
\usepackage{tikz-cd}
\usepackage{comment}
\usepackage{thmtools}
\usepackage{setspace}
\usepackage{mathtools}
\usepackage{enumitem}
\usepackage{thm-restate}

\usepackage{caption}
\usepackage{subcaption}
\usepackage{float}

\newtheorem{lemma}{Lemma}[section]
\newtheorem{lemma*}{Lemma}
\newtheorem{theorem}[lemma]{Theorem}
\newtheorem{proposition}[lemma]{Proposition}

\newtheorem{claim*}{Claim}

\newtheorem{defn}[lemma]{Definition}

\theoremstyle{definition}
\newtheorem{remark}[lemma]{Remark}

\theoremstyle{plain}

    \newtheoremstyle{TheoremNum}
        {\topsep}{\topsep}              
        {\itshape}                      
        {}                              
        {\bfseries}                     
        {.}                             
        { }                             
        {\thmname{#1}\thmnote{ \bfseries #3}}
    \theoremstyle{TheoremNum}

\newcommand{\C}{{\mathbb C}}
\newcommand{\F}{{\mathbb F}}

\newcommand{\Q}{{\mathbb Q}}
\newcommand{\R}{{\mathbb R}}
\newcommand{\Z}{{\mathbb Z}}

\DeclareMathOperator{\tr}{tr}

\DeclareMathOperator{\Hom}{Hom}

\DeclareMathOperator{\Spec}{Spec}

\DeclareMathOperator{\SL}{SL}

\numberwithin{equation}{section}
\numberwithin{table}{section}
\title{Finite covers and strict boundary slopes of cusped hyperbolic 3-manifolds}
\author{Tam Cheetham-West}
\address{Department of Mathematics \\ Yale University \\ New Haven, CT, 06511}
  \email{tamunonye.cheetham-west@yale.edu}

\author{Youheng Yao}
\address{Department of Mathematics \\ Yale University \\ New Haven, CT, 06511}
  \email{youheng.yao@yale.edu}

\begin{document}
\pagestyle{plain}
\maketitle
\begin{abstract}
    We prove that if two cusped hyperbolic $3$-manifolds admit a regular isomorphism between the profinite completions of their fundamental groups, then they share the same $A$-polynomial and their strongly detected boundary slopes match up. 
\end{abstract}
\bibliographystyle{alpha}

\section{Introduction}
Let $M$ be an orientable, compact, irreducible $3$-manifold. Recently, the problem regarding determining geometric and topological properties of $M$ using the profinite completion $\widehat{\pi_1(M)}$ of $\pi_1(M)$ has drawn significant attention. 
In this paper, we focus on the $A$-polynomial introduced in \cite{CCGLS}, which is a two-variable polynomial associated to an one-cusped $3$-manifold $M$ and encodes rich geometric and topological information about $M$. When $M$ is hyperbolic one-cusped $3$-manifold, we show that regular isomorphisms of profinite completions can only exist between cusped 3-manifolds that have the same $A$-polynomial. For example, we show

\begin{theorem}[{\Cref{pro:A_poly_profinite_inv}}]
\label{thm_A:A_poly}
    Let $M,N$ be one-cusped finite-volume hyperbolic 3-manifolds with $\Phi:\widehat{\pi_1(M)}\cong \widehat{\pi_1(N)}$ a regular isomorphism. Then they have the same $A$-polynomial $A_0^{M}(l,m) = A_0^{N}(l,m)$ which respects the peripheral structure.
\end{theorem}

The definition of the $A$-polynomial depends on the choice of a basis of the peripheral subgroup of $\pi_1(M)$. We say that the above identification {\it respects the peripheral structure} if there is an induced isomorphism $\psi:\pi_1(\partial M)\to \pi_1(\partial N)$ so that the $A$-polynomials are considered under bases $\langle l,m\rangle$ of $\pi_1(\partial M)$ and $\langle \psi(l),\psi(m)\rangle$ of $\pi_1(\partial M)$.
\medbreak
An isomorphism $\Phi:\widehat{\Gamma}\to\widehat{\Delta}$ between the profinite completions of two finitely generated, residually finite groups $\Gamma$ and $\Delta$ gives an isomorphism $\Phi_{ab}:\widehat{H}_1(\Gamma,\Z)\to \widehat{H}_1(\Delta,\Z)$ between the completions of their abelianizations. Boileau-Friedl \cite{BF} introduced the notion of a {\it regular} isomorphism to mean a profinite isomorphism $\Phi:\widehat{\Gamma}\to\widehat{\Delta}$ such that $\Phi_{ab}$ is induced by some isomorphism of the discrete groups $h:H_1(\Gamma,\Z)\to H_1(\Delta,\Z)$. They showed that regular isomorphisms only exist between profinite completions of knot complements in $S^3$ that have the same Alexander polynomial. 
\medbreak
\begin{remark}
    After completing this work, we learned that Xiaoyu Xu has established the regularity of profinite isomorphisms for cusped hyperbolic 3-manifold groups. Xu's work reproves Theorem 1.1 unconditionally for knots in the 3-sphere. 
\end{remark}
\medbreak
In his seminal work, Liu \cite{Y} showed that any profinite isomorphism $\Phi:\hat{\Gamma}\to\hat{\Delta}$ for $\Gamma,\Delta$ finite covolume Kleinian groups is {\it $\hat{\Z}-$regular}, meaning in this case that $\Phi_{ab}=m_\mu\circ \hat{h}$ where $m_\mu$ is the multiplication by a unit $\mu\in \hat{\Z}^\times$ and $h:H_1(\Gamma,\Z)\to H_1(\Delta,\Z)$ is an isomorphism. Using $\hat{\Z}-$regularity, Liu showed that every profinite isomorphism $\Phi:\hat{\Gamma}\to\hat{\Delta}$ of finite covolume Kleinian groups induced a Thurston-norm preserving linear isomorphism on cohomology. 

\medbreak

The main ingredient of this paper is that there is an identification of the $SL(2,\bar{\F}_p)$ representations for finitely generated residually finite groups with the same profinite completion. Together with the peripheral correspondence of the torus cusps studied in \cite{xuDehnfillings}, we show that the mod $p$ $A$-polynomials are the same under regular isomorphisms on profinite completions for all $p$. Then, we improve a result in \cite{ptilmanngarden} which connects Culler-Shalen theory and $A$-polynomials over characteristic $0$ and $p$. 
\medbreak
Based on this, we show that the peripheral correspondence in \cite{xuDehnfillings} restricts to the set of strongly detected slopes.
\begin{theorem}\label{thm_B:detected_slopes}
    Let $N_1$ be a 1-cusped hyperbolic 3-manifold, and let $N_2$ be a 3-manifold with $\Phi:\widehat{\pi_1(N_1)}\cong\widehat{\pi_1(N_2)}$ a regular isomorphism. Then there is a bijection between the strongly detected boundary slopes of $N_1$ and $N_2$. The Xu (Theorem A \cite{xuDehnfillings}) correspondence of slopes restricts to this bijection on strongly detected boundary slopes. 
\end{theorem}
Using the eigenvalue variety defined in \cite{Tillmann_boundary_slopes}, we can generalise \Cref{thm_B:detected_slopes} to the multi-cusped case, see \Cref{multcusps} in Section 5. As an application, we use results of Ni and Zhang \cite{Ni_Zhang_detection} to show that a family of Eudave-Mu\~noz knots are distinguished from all compact 3-manifolds by regular profinite isomorphisms.
\begin{theorem}\label{NiZhangtheorem}
    Let $K = k(\ell_*,-1,0,0)$ be an Eudave-Mu\~noz knot with $l_*>1$ and let $M = S^3\backslash K$. If $\Phi:\widehat{\pi_1(N)}\cong\widehat{\pi_1(M)}$ is a regular isomorphism, then $M\cong N$. 
\end{theorem}

\subsection{Organization}
Section 2 introduces preliminaries regarding profinite completion, representation variety and $A$-polynomial. In section 3, we prove that there is an isomorphism on the $SL(2,\overline{\mathbb{F}}_p)$ character variety given profinite isomorphism. Section 4 is devoted to prove \Cref{thm_A:A_poly} and \Cref{thm_B:detected_slopes}. In section 5, we show a multi-cusp version of \Cref{thm_B:detected_slopes}. We show some profinitely rigid examples using the previous sections in the last section.

\subsection{Acknowledgements}
The authors thank David Bai for suggesting the proof of \Cref{lem:polynomial_reduction}. The second author thanks Yair Minsky for helpful discussion.

\section{Preliminaries}

\subsection{Profinite completions of discrete groups}
For a countable, finitely generated, residually finite group $\Gamma$, the profinite completion $\hat{\Gamma}$ is the inverse limit of the inverse system of finite quotients of $\Gamma$. There is a natural injective homomorphism $\Gamma\hookrightarrow\hat{\Gamma}$, and for every homomorphism of discrete groups $\phi:\Gamma\to\Delta$ there is a natural extension $\hat{\phi}:\hat{\Gamma}\to\hat{\Delta}$ that fits into the commutative diagram
\[
\begin{tikzcd}
  \Gamma \arrow[r, "\phi"] \arrow[d]
    & \Delta \arrow[d] \\
  \hat{\Gamma} \arrow[r, "\hat{\phi}"]
&  \hat{\Delta} \end{tikzcd}
\]
where the vertical maps are canonical injections of the groups into their profinite completions. For every finite group $G$, extension gives a bijection between $\Hom(\Gamma,G)$, the set of homomorphisms from $\Gamma$ to $G$ and the set $\Hom(\hat{\Gamma},G)$, the set of abstract homomorphisms from $\hat{\Gamma}$ to $G$. As a consequence of a deep theorem of Nikolov-Segal \cite{ns}, all abstract homomorphisms from $\hat{\Gamma}$ to finite groups are continuous. 

\subsection{Character varieties}
Let $\Gamma$ be a finitely generated group with generators $\gamma_1,\cdots,\gamma_n$. Let $\bar{\F}_p$ be the algebraic closure of the finite field $\F_p$.
\medbreak
Fix an algebraically closed field $\F$. The $SL(2,\F)$-\textbf{representation variety} $R(\Gamma,\F) = \Hom(\Gamma, SL(2,\F))$ is the space of representations $\rho:\Gamma\to SL(2,\F)$. $R(\Gamma,\F)$ can be identified with its image under 
\begin{align*}
    R(\Gamma,\F) \to (SL(2,\F))^n\subset \F^{4n}\\
    \rho\mapsto (\rho(\gamma_1),\cdots,\rho(\gamma_n))
\end{align*}
relators of $\Gamma$ and determinant conditions of $SL(2,\F)$ give polynomial equations which define $R(\Gamma,\F)$ as an affine algebraic set.
\medbreak
The group $SL(2,\F)$ acts algebraically on $R(\Gamma,\F)$ by conjugation and we say that $\rho,\sigma\in R(\Gamma,\F)$ are closure-equivalent if the Zariski closures of their orbits intersect. We denote the quotient space of this equivalence relation $\sim$ by as 
\[X(\Gamma,\F) \coloneqq R(\Gamma,\F)/\sim \]
This is called the $SL(2,\F)$-\textbf{character variety} of $\Gamma$. $X(\Gamma,\F)$  is an affine algebraic set and its coordinate ring is generated by finitely many trace functions $tr_\gamma: X(\Gamma,\F)\to\F;\chi\mapsto \chi(\gamma)$.

\begin{remark}
    When $\Gamma = \pi_1(M)$ for a $3$-manifold $M$, we simplify the notation of $R(\Gamma,\F)$ and $X(\Gamma, \F)$ to $R(M,\F)$ and $X(M, \F)$ respectively for concision.
\end{remark}

\subsection{Boundary slopes}
We summarize relevant results regarding Culler-Shalen theory on constructing essential surfaces in \cite{CullerShalen} using $SL(2,\C)$ character varieties and introduced the notation of strongly detected slopes. See \cite{ptilmanngarden} for a detailed account for the case over $SL(2,\F)$.
\medbreak
Suppose that $M$ is an orientable, compact, irreducible $3$-manifold with torus boundary components $T_1,\cdots,T_n$. If an essential surface $S$ in $M$ has nonempty boundary, $\partial S$ intersects some $T_i$ with $n_i$ parallel copies of a simply closed curve $\alpha_i$ in $T_i$. This curve can be represented by a primitive element $\alpha_i = {p_i}\mathcal{M}_i+{q_i}\mathcal{L}_i$ for coprime $p_i,q_i\in\mathbb{Z}$. Then $\partial S$ can be identified by a point \[(n_1 p_1,n_1 q_1,n_2 p_2, n_2 q_2\cdots,n_h p_h, n_h q_h)\in\mathbb{Z}^{2h}-\{0\}\]
The above point is viewed as an element in $\mathbb{R}P^{2h-1}$. We further ignore the orientation of $\alpha_i$ in $T_i$, that is, 
${p_i}\mathcal{M}_i+{q_i}\mathcal{L}_i$ is identified with $-{p_i}\mathcal{M}_i-{q_i}\mathcal{L}_i$. This induces an action of $\mathbb{Z}_2^{h-1}$ on  $\mathbb{R}P^{2h-1}$. Note that $\mathbb{R}P^{2h-1}/\mathbb{Z}_2^{h-1}\cong S^{2h-1}$. For essential surfaces with nonempty boundary in $M$, their boundary curves are identified with the \textbf{projectivised boundary curve coordinates} \[[n_1 p_1,n_1 q_1,n_2 p_2, n_2 q_2\cdots,n_h p_h, n_h q_h]\in\mathbb{R}P^{2h-1}/\mathbb{Z}_2^{h-1}\]

When $M$ has a single boundary torus $T$ and $S$ is an essential surface in $M$ with nonzero boundary. We say that $\partial S$ determines a boundary slope $r = \frac{p}{q}\cup\{\infty\}$.
\medbreak

An ideal point $\xi$ of a curve in $X(M,\F)$ gives a field $\F$ with a discrete rank $1$ valuation. The group $SL(2,\F)$ acts on a Bass-Serre tree $T_\xi$ and we can pull back to an action of $\pi_1(M)$ using the tautological representation. Such action produces a dual essential surface $S$ via Stallings' construction in \cite{Stallings}. An essential surface $S$ which arises in this way is said to be detected by the ideal point $\xi$. 
\medbreak
When $S$ has nonempty boundary, the boundary slope $r$ of $S$ is called a \textbf{strongly detected slope} of $M$ (over $SL(2,\F)$) if there is such ideal point $\xi$ such that no closed essential surface in $M$ is detected by $\xi$.

\subsection{A-polynomials}
The A–polynomial was introduced in \cite{CCGLS} over $SL(2,\C)$ and investigated in \cite{ptilmanngarden} for the case over $SL(2,\bar{\F}_p)$. We recall the definition and some results here. 
\medbreak
Let $M$ be a compact $3$-manifold with boundary a single torus $T$ with meridian $\mathcal{M}$ and $\mathcal{L}$. We can identify $\mathcal{M}$ and $\mathcal{L}$ with their images under the inclusion $\pi_1(T)\hookrightarrow\pi_1(M)$. Let $R_U(M,\F)$ be the set of representations $\rho\in R(M,\F)$ such that $\rho(\mathcal{M})$ and $\rho(\mathcal{L})$ are upper-triangular matrices. This is a closed subset of $R_U(M,\F)$ and every representation in $R(M,\F)$ is conjugate to an element in $R_U(M,\F)$. For $\rho\in R_U(M,\F)$, let 
\begin{align}\label{eq:Upper_tri_A_poly}
    \rho(\mathcal{M}) = \begin{pmatrix}
    m& *\\
    0 & m^{-1}
    \end{pmatrix},\,\,\,\rho(\mathcal{L}) = \begin{pmatrix}
    l& *\\
    0 & l^{-1}
    \end{pmatrix}
\end{align}
\medbreak
The eigenvalue map is defined as 
\begin{align*}
    e: R_U(M,\F)&\to(\F-\{0\})^2\\
    \rho&\mapsto (m,l)
\end{align*}
The Zariski closure of the image of $e$ is called the $SL(2,\F)$ \textbf{eigenvalue variety} of $M$, which is denoted by $E_2(M)$. Deleting all $0$-dimensional components from $E_2(M)$, the remaining part is neccesarily generated by a single polynomial in $\F[m^{\pm1},l^{\pm1}]$, which is called the $A$-\textbf{polynomial} of $M$. If $\F = \C$ or $\F = \bar{\F}_p$, we denote the A-polynomial as $A_0^M(l,m)$ and $A_p^M(l,m)$ respectively. 

\begin{remark}
    The $A$-polynomials can always be chosen so that $A_0\in \Z[l,m]$ and $A_p\in\Z_p[l,m]$. See \cite[Lemma 36]{ptilmanngarden} for details.
\end{remark}
\medbreak
The $A$-polynomial carries topological properties for the detected essential surfaces from $X(M,\F)$. We note the following connection between the strongly detected slopes and sides of the Newton polygon of the $A$-polynomial.

\begin{lemma}[{\cite{CCGLS} for char $0$} and \cite{ptilmanngarden} for char $p$]\label{bslopesarebslopes}
    Let $\F$ be an algebraically closed field of characteristic $p$. The boundary slopes of the Newton polygon for $A_p^M$ are the strongly detected boundary slopes for $X(M,\F)$.
\end{lemma}
\medbreak
Based on the above lemma, the following lemma connects the strongly detected slopes over $SL(2,\C)$ and over $SL(2,\F)$.
\begin{lemma}[{\cite[Theorem 47]{ptilmanngarden}}]\label{TG_slope_theorem}
    Let $M$ be an orientable cusped finite-volume $3$-manifold. For all but finitely many primes $p$, the boundary curve of an essential surface in $M$ is strongly detected by $X(M,\F)$ for $\F$ an algebraically closed field of characteristic $p$ if and only if it is
strongly detected by $X(M,\C)$.
\end{lemma}
\medbreak
The eigenvalue variety, which is a generalization of the $A$-polynomial to $3$-manifolds with multiple boundary components, is treated in \cite{Tillmann_boundary_slopes} and \cite{ptilmanngarden}.
Given a presentation of $\pi_1(M)$ with generators $\gamma_1,\cdots,\gamma_n$, $R(M,\F)$ is generated by an ideal $J =\langle f_1,\cdots,f_k\rangle$ with $\Z_p[g_{11},..., g_{n4}]$. Recall that the trace function $I_\gamma(\rho) = \tr\rho(\gamma)$ is in $\Z_p[g_{11},..., g_{n4}]$ for any $\gamma\in\pi_1(M)$. In the ring $\Z_p[g_{11},..., g_{n4},m_1^{\pm 1},l_1^{\pm1},\cdots,m_h^{\pm1},l_h^{\pm1}]$, we define the following equations:
\begin{align*}
    I_{\mathcal{M}_{i}}&= m_i+m_i^{-1}\\
    I_{\mathcal{L}_{i}}&= l_i+l_i^{-1}\\ 
    I_{\mathcal{M}_i\mathcal{L}_i} &= m_il_i+m_i^{-1}l_i^{-1}
\end{align*}
Let $R_E(M,\F)$ be the algebraic set in $\F^{4n}\times \F^{2h}$ generated by $J$ together with above $3h$ equations. 
\begin{defn}
The $SL(2,\mathbb{\F})$-\textbf{eigenvalue variety} $E_2(M,\F)$ of $M$ is the Zariski closure of the image of $R_E(M,\F)$ under the projection map
\begin{align*}
p: R_E(M,\F)&\to (\F-\{0\})^{2h}\\
    (g_{11},..., g_{n4},m_1,\cdots,l_h)&\mapsto (m_1,\cdots,l_h)
\end{align*}
$E_2(M,\F)$ is defined by an ideal $K_\F^M \in \F[m_1,l_1,\cdots,m_h,l_h]$.
\end{defn}

\subsection{The logarithmic limit set of a character variety}\label{subsec:log_limit_set}
The boundary slopes detected by the character variety are recovered by the \textbf{logarithmic limit set} of the eigenvalue variety. We list three definitions following \cite{B_logarithmic_limit_set}.
\medbreak
Let $I$ be an ideal in $\F[x_1^{\pm},\cdots,x_n^{\pm}]$ and  $V = V(I)$ be the variety in $(\F\backslash\{0\})^n$. 
\begin{enumerate}
    \item An absolute value function on $\F$ is a function $|\cdot|: \F\to\R_{\geq 0}$ such that $|0| = 0$, $|1| = 1$, $|xy| = |x||y|$,
$|x+y|\leq|x|+|y|$ and there exists $x\in \F$ such that $0 < |x| < 1$. $V_{log}$ is the set of limit points on $S^{n-1}$ of the set 
    \[\Bigl\{\dfrac{(\log |x_1|,\cdots,\log |x_n|)}{\sqrt{1+\sum_{i=1}^n (\log |x_i|)^2}}\mid (x_1,\cdots,x_n)\in V\Bigr\}\]
    \item The \textbf{tropical variety} $V_{val}$ is the set of $n$-tuples
    \[ (-v(x_1),\cdots,-v(x_n))\]
    where $v$ is any real-valued valuation on $\mathbb{\F}[x_1^{\pm},\cdots,x_n^{\pm}]/I$ such that $\sum_{i=1}^n v(x_i)^2 = 1$.
    \item Let the support $s(f)$ of $f\in \mathbb{\F}[x_1^{\pm},\cdots,x_n^{\pm}]$ to be set of points $\alpha =(\alpha_1,\cdots,\alpha_n)\in \mathbb{Z}^n$ such that the monomial $X^\alpha = X^{\alpha_1}\cdots X^{\alpha_n}$ appears in $f$. Let $V_{sph}$ be the set of points $\xi\in S^{n-1}$ such that for any non-zero $f\in I$, the maximal value of $\xi\cdot \alpha$ is achieved at least twice. Geometrically, 
    \[V_{sph} = \cap_{f\neq0\in I}\text{Sph}(\text{Newt}(f))\] is the intersection of spherical duals of the Newton polytopes of all nonzero elements $f$ in $I$.
\end{enumerate}
\medbreak
It is shown in \cite{B_logarithmic_limit_set} that $V_\infty  = V_{val} = V_{sph}$ and $V_{log}\subseteq V_{\infty} $ is a closed subset.
\medbreak
We denote $V_\infty(K_\F^M)$ as the logarithmic limit set of the eigenvalue variety $E_2^p(M)$. By symmetry, for any ideal point $(m_1,l_1,\cdots,m_i,l_i,\cdots,m_h,l_h)\in E_2^p(M)$, clearly the point $(m_1,l_1,\cdots,m_i^{-1},l_i^{-1},\cdots,m_h,l_h)$ is also in $E_2^p(M)$. Hence $(x_1,x_2,\cdots,x_{2i-1},x_{2i},\cdots ,x_{2h-1},x_{2h})$ is in $V_\infty(K_\F^M)$ implies that \linebreak $(x_1,x_2,\cdots,-x_{2i-1},-x_{2i},\cdots x_{2h-1},x_{2h})$ is in $V_\infty(K_\F^M)$. This induces a quotient map \[\mathbb{R}P^{2h-1}\to \mathbb{R}P^{2h-1}/\mathbb{Z}_2^{h-1}\cong S^{2h-1}\]
and extends to a map $\Psi: S^{2h-1}\to S^{2h-1}$. Let 
\[ T = \begin{pmatrix} 0 & 1\\ -1 & 0\end{pmatrix}\]
and $T_h$ be the block diagonal matrix consisting of $h$ copies of $T$.

\subsection{Dehn fillings}
Let $M$ be a irreducible, orientable compact $3$-manifold with boundary components consisting of tori $\partial_1 M, \partial_2 M,\cdots,\partial_n M$. Let $c_i$ be a slope of $\partial_i M$ for $1\leq i\leq n$ with empty slopes allowed. We denote the Dehn filled manifold obtained by filling $\partial_iM$ along $c_i$ as $M_{c_i}$. The main result in \cite{xuDehnfillings} concerning the profinite completions of the Dehn-filled manifolds is as follows.

\begin{lemma}[{\cite[Theorem A]{xuDehnfillings}}]\label{xustructuretheorem}
    Suppose $M$ and $N$ are orientable cusped finite-volume hyperbolic 3-manifolds and $\Phi:\widehat{\pi_1(M)}\to\widehat{\pi_1(N)}$ is a profinite isomorphism. Then
    \begin{enumerate}
        \item $\Phi$ respects the peripheral structure, i.e. there is a one-to-one correspondence between the boundary components of $M$ and $N$, which we simply denote as $\partial_iM\leftrightarrow\partial_iN$, so that $\Phi(\overline{\pi_1(\partial_iM}))$ is a conjugate of $\overline{\pi_1(\partial_iN)}$ in $\widehat{\pi_1(N)}$.
        \item Under the correspondence between the cusps $\partial_i M$ of $M$ and $\partial_{i}N$ of $N$, there is an isomorphism $\psi_i:\pi_1\partial_i M\to \pi_1\partial_i N$ such that, for any boundary slopes $(c_i)$ on $\partial M$ (allowing empty slopes), there is an isomorphism $\widehat{\pi_1(M_{c_i}})\cong\widehat{\pi_1(N_{\psi(c_i)})} $ that respects the peripheral structure.
        \item At the level of fundamental group, $\exists g_i\in\widehat{\pi_1(N)}$ and a profinite unit $\mu_i\in\hat{\Z}$ such that $m_{\mu_i}\circ\psi_i=  C_{g_i}\circ\Phi:\overline{\pi_1(\partial_iM)}\to\overline{\pi_1(\partial_iN)}$ where $m_{\mu_i}$ is multiplication by $\mu_i$ and $C_{g_i}$ is conjugation by $g_i$.
    \end{enumerate}
\end{lemma}

\section{Detecting the characteristic $p$ representation variety}
For the rest of this discussion, $\F$ will be an algebraically closed field with prime characteristic $p$. The following lemmas are well known (see \cite[Lemma 5.8, 5.9]{rapoport}, for example). We include the proofs for convenience.
\begin{lemma}\label{replemma}
    For $\Gamma,\Delta$ finitely generated, residually finite groups with $\hat{\Gamma}\cong\hat{\Delta}$, then $R(\Gamma,\F)=R(\Delta,\F)$ as sets.
    \begin{proof}
        Fix an identification $\Phi:\hat{\Gamma}\to\hat{\Delta}$. The image of a finitely generated group in $SL(2,\F)$ is finite. Thus, every representation $\rho:\Gamma\to SL(2,\F)$ has a unique extension $\hat{\rho}:\hat{\Gamma}\to SL(2,\F)$. The homomorphism $\hat{\rho}\circ \Phi^{-1}$ restricted to $\Delta$ has the same image as $\rho$ in $SL(2,\F)$. Thus $R(\Gamma,\F)\subset R(\Delta,\F)$. For the reverse inclusion, given a homomorphism $\varphi:\Delta\to SL(2,\F)$, $\hat{\varphi}\circ\Phi$ restricted to $\Gamma$ is a homomorphism with the same finite image as $\varphi$ since $\Delta<\hat{\Delta}\cong\hat{\Gamma}$ is dense. 
 \end{proof}   
 \end{lemma}
\section{Detecting the A-polynomial}

\begin{lemma}\label{mod_p_A-polynomial}
    Let $N_1,N_2$ be 1-cusped finite-volume hyperbolic 3-manifolds with a regular isomorphism $\widehat{\pi_1(N_1)}\cong \widehat{\pi_1(N_2)}$. Then $A_p^{N_1}(L,M)\cong A_p^{N_2}(L,M)$ for all primes $p$. 
    \begin{proof}
    We fix an identification $\Phi:\widehat{\pi_1(N_1)}\to\widehat{\pi_1(N_2)}$. Let $\F = \bar{\F}_p$ for some fixed prime $p$.
    \medbreak
    We show that the bijection from Lemma~\ref{replemma} restricts to a bijection $R_U(N_1,\F)$ and $R_U(N_2,\F)$. For any $\rho\in R_U(N_2,\F)$, it extends uniquely to $\hat{\rho}:\widehat{\pi_1(N_2)}\to SL(2,\F)$ whose images of the peripheral subgroup are upper triangular matrices. By 
    Lemma~\ref{xustructuretheorem} (1), $\overline{\pi_1(\partial N_2)} = g\overline{\pi_1(\partial N_1)}g^{-1}$ for some $g\in \widehat{\pi_1(N_1)}$. Thus, $\overline{\pi_1(\partial N_1)}$ has images in upper triangular matrices under $\hat{\rho}\circ\Phi$ so it restricts to an element in $R_U(N_1,\F)$. Similar to Lemma~\ref{replemma}, this induces an isomorphism between $R_U(N_1,\F)$ and $R_U(N_2,\F)$.
    \medbreak
    Let $\langle m,l\rangle$ be a basis for $\pi_1(\partial N_1)$. By Lemma~\ref{xustructuretheorem} (3), we have an isomorphism $\psi:\pi_1(\partial N_1)\to\pi_1(\partial N_2)$  such that the restriction of $\Phi$ to $\overline{\pi_1(\partial N_1)}$ satisfying
    \[m_{\mu}\circ \psi = C_g\circ \Phi:\overline{\pi_1(\partial N_1)}\to\overline{\pi_1(\partial N_2)}\]
    When $\Phi$ is regular, we have $\mu = \pm 1\in \hat{\Z}$.
    Consider the basis $\langle m',l'\rangle=\langle\psi(m),\psi(l)\rangle$   of $\pi_1(\partial N_2)$; $\hat{\rho} \circ \Phi$ maps $m,l$ to some conjugates of $\hat{\rho}(m')^{\pm 1},\hat{\rho}(l')^{\pm 1}$. Hence the eigenvalue maps $e_{N_1}=e_{N_2}$ are the same with respect to the basis $\langle m,l\rangle$ of $\pi_1(\partial N_1)$ and $\langle m',l'\rangle$ of $\pi_1(\partial N_2)$.
    \medbreak
    
    Thus, the eigenvalue maps have the same image and Zariski closure and by Lemma 36 \cite{ptilmanngarden}, we have $A_p^{N_1}(L,M)=A_p^{N_2}(L,M)\in \Z_p[m^\pm,l^\pm]$.
    \end{proof}
\end{lemma}
\medbreak
We now improve a result of \cite{ptilmanngarden} so that the $A$-polynomial over $\SL(2,\C)$ is determined by the mod-$p$ $A$-polynomials. This shows that the $A$-polynomial is a regular profinite invariant in \Cref{pro:A_poly_profinite_inv}. To do so, we need the following lemma using some algebraic geometry. We say that a polynomial $f$ over a field $k$ is \textit{reduced} if it has no multiple factor when considered over $\bar{k}$.

\begin{lemma}\label{lem:polynomial_reduction}
If $f\in\Z[m,l]$ is reduced when considered over $\bar{\Q}$, then for all but finitely many prime $p$, $\overline{f} = f\bmod p$ is reduced over $\bar{\F}_p$.
\begin{proof}
    Denote $\varphi:\Spec \Z[m,l]/(f)\to \Spec \Z$ as the structure morphism of the natural map from $\Z$ to $\Z[m,l]/(f)$. We aim to apply the following lemma in the stack project:
\begin{lemma}[{\cite[\href{https://stacks.math.columbia.edu/tag/0578}{Lemma 0578}]{stacks-project}}]\label{lem:scheme_morphism}
    Let $\varphi:X\to Y$ be a morphism of schemes. Assume that
    \begin{enumerate}
        \item $Y$ is irreducible with generic point $\eta$,
        \item $X_\eta$ is geometrically reduced,
        \item $f$ is of finite type.
    \end{enumerate}
    Then there exists a nonempty open subscheme $V\subset Y$ such that $X_V\to V$ has geometrically reduced fibers.
\end{lemma}
Clearly $Y =\Spec \Z$ is irreducible with the unique generic point $(0)$. Note that \[X_{(0)}\times_{\Spec\Q}\Spec\bar{\Q}\cong \Spec \bar{\Q}[m,l]/(f)\]
$f$ is reduced implies that $X_{(0)}\times_{\Spec\Q}\Spec\bar{\Q}$ is reduced. $\varphi$ is clearly of finite type. Thus, there is an nonempty open subset $V$ of $\Spec\Z$ satisfying the conclusion in \Cref{lem:scheme_morphism}. $V$ must be of form \[V = \Spec \Z \backslash \{(p_i)\mid \text{prime }p_i \text{ for } i = 1,\cdots,n \}\]Then, for all but finitely many prime $p$,
\[X_{p}\times_{\Spec\F_p}\Spec\bar{\F}_p\cong \Spec \bar{\F}_p[m,l]/(\bar{f})\]
is reduced where $\bar{f} = f\bmod p$ and $\overline{f}$ is reduced over $\bar{\F}_p$. 
\end{proof}
\end{lemma}

\begin{proposition}\label{pro:A_poly_profinite_inv}
    Let $N_1,N_2$ be regularly profinitely equivalent 1-cusped finite-volume hyperbolic 3-manifolds. Then $A_0^{N_1}(L,M)\cong A_0^{N_2}(L,M)$. 
    \begin{proof}
Let $M$ be an one-cusped $3$-manifold and denote its A-polynomials as $A_0$ and $A_p$. For a prime $p$, denote $\overline{A_0} = A_0 \bmod p$.
By \cite[Lemma 40]{ptilmanngarden}, for all but finitely many prime $p$,
\[\overline{A_0}(l,m) = g(l,m)A_p(l,m)\]
where $g\in \Z_p[l,m]$ consists of factors appearing in $A_p$. Applying \Cref{lem:polynomial_reduction}, $\overline{A_0}$ is reduced so that $g(l,m) = 1$ for all but finitely many $p$.
\medbreak
If $\widehat{\pi_1(N_1)}\cong \widehat{\pi_1(N_2)}$, the mod $p$ $A$-polynomials $A_p^{N_1}(l,m), A_p^{N_2}(l,m)$ agree for all prime $p$. Then $\overline{A_0^{N_1}}(l,m) = \overline{A_0^{N_2}}(l,m)$ for all but finitely many prime $p$, this forces the result.
\end{proof}
\end{proposition}
\begin{theorem}\label{onecusp}
    Let $N_1$ be a 1-cusped hyperbolic 3-manifold, and let $N_2$ be a 3-manifold with a regular isomorphism $\Phi:\widehat{\pi_1(N_1)}\cong\widehat{\pi_1(N_2)}$. Then there is a bijection between the strongly detected boundary slopes of $N_1$ and $N_2$. The Xu ( \cite[Theorem A]{xuDehnfillings}) correspondence of slopes restricts to this bijection on strongly detected boundary slopes. 
    \begin{proof}
    By Theorem 47 of \cite{ptilmanngarden}, a boundary slope is strongly detected in characteristic $0$ if and only if it is strongly detected in characteristic $p$ for all but finitely many primes $p$. 
    \medbreak By Lemma~\ref{mod_p_A-polynomial}, we show that the characteristic $p$ $A-$polynomial $A_p$ is a profinite invariant, i.e. that both cusped 3-manifolds have the same characteristic $p$ $A-$polynomial. Since the boundary slopes of the Newton polygon for $A_p^{N_1}(L,M)=A_p^{N_2}(L,M)$ are the strongly detected boundary slopes for $X(\pi_1(N_1),\F)\cong X(\pi_1(N_2),\F)$  by Theorem 37 \cite{ptilmanngarden}, $N_1,N_2$ have the same set of strongly detected boundary slopes in characteristic $0$.
    \medbreak
    It remains to show that the bijection between strongly detected boundary slopes is a restriction of Xu's correspondence in Lemma~\ref{xustructuretheorem}. For a fixed identification $\Phi:\widehat{\pi_1(N_1)}\to\widehat{\pi_1(N_2)}$, and for a strongly detected boundary slope $c\in\partial N_1$ corresponding to a primitive element in $\Z^2\cong\pi_1(\partial N_1)$, $c=m^\alpha l^\beta$ where $\{m,l\}$ is a basis for $\Z^2$ and $\alpha,\beta\in\Z$. There is an element $g\in\widehat{\pi_1(N_2)}$ such that $g\Phi(\overline{\pi_1(\partial N_1)})g^{-1}=\pi_1(\partial N_2)$. Let the ``conjugation by $g$" map be denoted by $C_g:\widehat{\pi_1(N_2)}\to\widehat{\pi_1(N_2)}$. There is a profinite unit $\mu\in\hat{\Z}^\times$ such $\mu\circ C_g\circ \Phi|_{\pi_1(\partial N_1)}$ has image $\pi_1(\partial N_2)$, and this is the isomorphism induced on $\pi_1$ by the map called $\psi_i$ in Lemma~\ref{xustructuretheorem} (b). We denote this isomorphism on the cusps by $\pi_1(\psi_i)$. The set $\{m',l'\}=\{\pi_1(\psi_i)(m),\pi_1(\psi_i)(l)\}$ is a basis for $\pi_1(\partial N_2)$ since $\pi_1\psi_i$ is an isomorphism and $\{m,l\}$ is a basis of $\Z^2\cong\pi_1(\partial N_1)$. To show that $m'^\alpha l'^\beta$ is a strongly detected boundary slope for $N_2$, we apply Lemma~\ref{mod_p_A-polynomial} to see that for all primes $p$, $A_p^{N_1}(L,M)=A_p^{N_2}(L,M)\in \Z_p[m^\pm,l^\pm]$ and therefore by Lemma~\ref{TG_slope_theorem} and Lemma~\ref{bslopesarebslopes}, $m'^\alpha l'^\beta$ is a strongly detected boundary slope for $N_2$. 
    \end{proof}
\end{theorem}

\section{The case of multiple cusps}
\noindent To handle the case of hyperbolic 3-manifolds with multiple boundary components, we show that the {\it logarithmic limit set} of the character variety over characteristic $p$ is a regular profinite invariant. 
\begin{lemma}\label{lem:eigenval_variety}
Let $M,N$ be finite-volume cusped hyperbolic 3-manifolds with $\Phi:\widehat{\pi_1(M)}\to\widehat{\pi_1(N)}$ a regular
 isomorphism. The identification $\Phi$ induces an equality of the characteristic $p$ eigenvalue varieties $E_2(M,\F)=E_2(N,\F)$ and therefore an equality of the logarithmic limit sets $V_\infty(K^\mathbb{F})$ of $E_2^p(M)$. 
\begin{proof}
Fix $\F=\bar{\F}_p$ for some prime $p$. Similar as the proof in \Cref{mod_p_A-polynomial}, there is an identification $\Phi_*: R(N,\F)\to R(M,\F)$ sending $\rho$ to $(\hat{\rho}\circ \Phi)\mid_{\pi_1(M)}$. In addition, for each cusp correspondence $\partial_i M\leftrightarrow\partial_iN$, the restriction of $\hat{\rho}\circ\Phi$ to $\pi_1(\partial_{i}M)$ is $\hat{\rho}\circ (C_{g_i}\circ m_{\mu}\circ \psi_i)$ for some $g_i\in \widehat{\pi_1(N)}$ and $\mu = \pm 1\in \hat{\Z}$. Fix a basis $\langle\mathcal{L}_i,\mathcal{M}_i\rangle$ of $\pi_1(\partial_i M)$ and $\langle\mathcal{L}_i',\mathcal{M}_i'\rangle = \langle\psi(\mathcal{L}_i'),\psi(\mathcal{M}_i')\rangle$ of $\pi_1(\partial_iN)$. The trace function $I_{\alpha}\in\F[R(M,\F)]$ maps to $I_{\psi_i(\alpha)}\in \F[R(N,\F)]$ for any $\alpha\in \partial_i(M)$, which implies $R_E(M,\F)=R_E(N,\F)$ by definition. This directly implies the equality $E_2(M,\F)=E_2(N,\F)$ and the equality of the logarithmic limit sets.

\end{proof}
\end{lemma}
\noindent Theorem~\ref{onecusp} is a warmup for
\begin{theorem}\label{multcusps}
    Let $M$ be a finite-volume hyperbolic 3-manifold with $n$ cusps. If $N$ is a 3-manifold with $\widehat{\pi_1(M)}\cong\widehat{\pi_1(N)}$ via a regular isomorphism, then there is a bijection between the strongly detected boundary curves of $M$ and $N$.
    \begin{proof}
    Let the torus cusps of $M$ be $\partial_1 M,\dots,\partial_nM$. By Theorem 9.1 \cite{WZ1} and Theorem 1 \cite{Chagas2016Hyperbolic3G}, $N$ is a hyperbolic finite-volume 3-manifold with $n$ cusps $\partial_1 N,\dots,\partial_n N$. Let $S\subset M$ be an embedded incompressible surface with strongly detected boundary slopes $s_1,\dots,s_n$ (some possibly empty).
    \medbreak
    By Theorem 45 \cite{ptilmanngarden}, there is an element $\zeta\in V_\infty(K_\F^M)$ (the logarithmic limit set) with rational coordinate ratios with $[s_i]=\varphi(T_n\zeta)$ where the matrix $T_n$ and $\varphi: S^{2n-1}\to S^{2n-1}$ are defined in \cref{subsec:log_limit_set}. Since $E_2(M,\F)$ and $E_2(N,\F)$ share the same logarithmic limit set by \Cref{lem:eigenval_variety}, we have $[s_i]$ is also a projectivised boundary curve detected by $E_2(N,\F)$. 
    \medbreak 
    We only need to remark that the coordinate $[s_i]$ is interpreted under the fixed basis $\langle\mathcal{L}_i,\mathcal{M}_i\rangle$ of $\pi_1(\partial_i M)$ and $\langle\mathcal{L}_i',\mathcal{M}_i'\rangle = \langle\psi_i(\mathcal{L}_i),\psi_i(\mathcal{M}_i)\rangle$ of $\pi_1(\partial_iN)$, where the $\psi_i$ are the isomorphisms $\psi_i:\pi_1\partial_i M\to\pi_1\partial_iN$ in Lemma~\ref{xustructuretheorem}.

    \end{proof}
\end{theorem}
\section{Examples}
In \cite[Theorem 1.2]{Ni_Zhang_detection}, they showed that an infinite family of Eudave-Mu\~noz knots $K = k(\ell,-1,0,0)$ with $\ell>1,\ell\in\Z$ can be detected using the knot Floer homology and the $A$-polynomial. We show that the profinite completions of fundamental groups of every knot in this family do not admit regular profinite isomorphisms to the profinite completions of the fundamental groups of any other non-homeomorphic compact 3-manifold group. Note that Xu in \cite{xuDehnfillings}
showed that another family $k(3,1,n,0)$ are profinitely rigid.
\begin{theorem}
Let $K = k(\ell_*,-1,0,0)$ be an Eudave-Mu\~noz knot with $l_*>1$ and let $M = S^3\backslash K$. If $\Phi:\widehat{\pi_1(N)}\cong\widehat{\pi_1(M)}$ is a regular isomorphism, then $M\cong N$. 
\end{theorem}
\begin{proof}
    Let $N$ be a compact orientable $3$-manifold with a regular isomorphism $\widehat{\pi_1(N)}\cong \widehat{\pi_1(M)}$. By \cite[Theorem C]{xuDehnfillings}, $N$ is a hyperbolic knot. 
    \medbreak
    For a knot complement $M = S^3\backslash K$, we can choose a meridian-longitude basis $\langle\mathcal{M},\mathcal{L}\rangle$ for $\pi_1\partial M\cong\Z^2$  such that $\mathcal{M}$ normally generates $\pi_1(M)$ and $\mathcal{L}$ is the unique null homologous slope that bounds a Seifert surface in $M$. 
    \medbreak
    By \cite[Theorem C]{xuDehnfillings}, the induced map $\psi:slp(M)\to slp(N)$ is $\pm 1$ under the stand meridian-longitude basis for both knots.
    \Cref{pro:A_poly_profinite_inv} implies that $M$ and $N$ have the same A-polynomial $A_0(l,m)$. \cite[Theorem 1.2]{Ni_Zhang_detection} shows that the knot Floer homology and $A$-polynomial distinguish $K$ among all knots in $S^3$ (in particular, among hyperbolic knots). Indeed, they only use the information of Seifert genus from the Floer homology. It is well-known that Seifert genus is a profinite invariant by \cite{BF}. Thus, \cite[Theorem 1.2]{Ni_Zhang_detection} gives us the desired result.
\end{proof}


Let $M = S^3\backslash K$ where $K = k(\ell,-1,0,0)$. Suppose that there is another hyperbolic cusped $3$-manifold $N$ profinitely isomorphic to $M$, then Since $M$ has exactly one half-integral toroidal slope $r$, we have the following result.

\begin{lemma}\label{lem:pro_eudave_knot}
    Let $M = S^3\backslash K$ where $K = k(l,m,n,p)$ is an Eudave-Mu\~noz knot. Suppose that $N$ is a $3$-manifold with $\widehat{\pi_1(M)}\cong\widehat{\pi_1(N)}$. Then $N$ is also an Eudave-Mu\~noz knot with a half-integral toroidal slope $r'$. In addition, the induced map $\psi:slp(M)\to slp(N)$ maps $r$ to $r'$.
    \begin{proof}
        By \cite[Theorem C]{xuDehnfillings}, $N$ is a hyperbolic knot and $\widehat{\pi_1(M_r)}\cong\widehat{\pi_1(N_{\psi(r)})}$. Then \cite[Lemma 5.4]{xuDehnfillings} implies that $\psi(r)$ is a half-integral toroidal slope of $N$. The result follows from the fact that Eudave-Mu\~noz knots are the only knots admitting one (and exactly one) non-integral toroidal Dehn surgery as shown in \cite{Gordon_Luecke}.
    \end{proof}
\end{lemma}


\bibliography{main}

\end{document}